\documentclass[11pt,leqno]{amsart}
\usepackage{color,amsmath,amsfonts,amssymb,amscd,amsthm,amsbsy,upref}
\numberwithin{equation}{section}

\newtheorem{thm}{Theorem}[section]
\newtheorem{lem}[thm]{Lemma}
\newtheorem{prop}[thm]{Proposition}

\newtheorem*{Gap}{Gap lemma}
\theoremstyle{definition}

\newtheorem{prob}[thm]{Problem}

\newtheorem{remark}[thm]{Remark}

\newtheorem*{rem}{Remark}

\def\N{{\mathbb N}}
\def\R{{\mathbb R}}

\def\R{{\mathbb R}}

\def\embed{\hookrightarrow}

\def\A{{\mathcal A}}
\def\F{{\mathcal F}}
\def\G{{\mathcal G}}
\def\ep{\varepsilon}
\newcommand{\con}{\!\smallfrown\!}
\newcommand{\Con}{\smallfrown}

\def\chix{{\raise.5ex\hbox{$\chi$}}}
\def\embed{\hookrightarrow}
\newcommand{\vertiii}[1]{{\left\vert\kern-0.25ex\left\vert\kern-0.25ex\left\vert #1 
    \right\vert\kern-0.25ex\right\vert\kern-0.25ex\right\vert}}
\begin{document}
\allowdisplaybreaks

\title{On spreading sequences and asymptotic structures}
\author{D. Freeman}
\author{E. Odell} \thanks{Edward Odell (1947-2013).}
\author{B. Sar\i}
\author{B. Zheng}

\address{Department of Mathematics and Computer Science\\
Saint Louis University , St Louis, MO 63103  USA}
\email{dfreema7@slu.edu}

\address{Department of Mathematics\\ The University of Texas at Austin\\
1 University Station C1200\\ Austin, TX 78712-0257}

\address{Department of Mathematics, University of North Texas, Denton, TX 76203-5017}
\email{bunyamin@unt.edu}

\address{Department of Mathematics, University of Memphis, Memphis, TN, 38152-3240}
\email{bzheng@memphis.edu}

\thanks{Research of the first, second, and fourth author was supported by the
  National Science Foundation.}
  \thanks{  Research of the first author was supported by grant 353293 from the Simon's foundation and the third author was
  supported by grant 208290 from the Simon's Foundation.}
\begin{abstract}
In the first part of the paper we study the structure of Banach spaces with a conditional spreading basis. The geometry of such spaces exhibit a striking resemblance to the geometry of James' space. Further, we show that the averaging projections onto subspaces spanned by constant coefficient blocks with no gaps between supports are bounded. As a consequence, every Banach space with a spreading basis contains a complemented subspace with an unconditional basis. This gives an affirmative answer to a question of H. Rosenthal. 

The second part contains two results on Banach spaces $X$ whose asymptotic structures are closely related to $c_0$ and do not contain a copy of $\ell_1$: 

i) Suppose $X$ has a normalized weakly null basis $(x_i)$ and every spreading model $(e_i)$ of a normalized weakly null block basis satisfies $\|e_1-e_2\|=1$. Then some subsequence of $(x_i)$ is equivalent to the unit vector basis of $c_0$. This generalizes a similar theorem of Odell and Schlumprecht, and yields a new proof of the Elton-Odell theorem on the existence of infinite $(1+\ep)$-separated sequences in the unit sphere of an arbitrary infinite dimensional Banach space.

ii) Suppose that all asymptotic models of $X$ generated by weakly null arrays are equivalent to the unit vector basis of $c_0$. Then $X^*$ is separable and $X$ is asymptotic-$c_0$ with respect to a shrinking basis $(y_i)$ of $Y\supseteq X$.
\end{abstract}
\maketitle

\baselineskip=18pt      

\section{Introduction}

A basic sequence $(x_i)$ in a Banach space is called {\em spreading} if it is equivalent to all of its subsequences. If, in addition, the sequence is unconditional then it is called {   \em  subsymmetric}. When $(x_i)$ is spreading and weakly null it is automatically suppression unconditional. { In  Section 2} we will focus most of our attention on spreading sequences that are not unconditional. A famous example is the boundedly complete basis of the James space $J$ and we shall see that much of the structure for $J$ holds more generally for Banach spaces with a conditional spreading basis. 
We observe that if $(e_i)$ is a normalized conditional spreading basis for $X$ then the difference sequence $(d_i)=(e_1, e_2-e_1, e_3-e_2, \ldots)$ is a skipped unconditional basis for $X$. This means that if $(x_j)$ is a normalized block basis of $(d_i)$ with ${\rm supp}(x_j)<i_j<{\rm supp}(x_{j+1})$ for some subsequence $(i_j)$ of $\N$, then $(x_j)$ is unconditional. Here ${\rm supp}(x_j)$ refers to the basis $(d_i)$, that is, {if} $x_j=\sum_i b^j_i d_i$ then ${\rm supp}(x_j)=\{i: b^j_i\neq 0\}$. It follows that, in the case $(e_i)$ is spreading but not weakly null, $\ell_1\not\hookrightarrow X$ ($\ell_1$ does not embed isomorphically into $X$) if and only if the difference basis $(d_i)$ is shrinking. Also we show that $c_0\not\hookrightarrow X$ if and only if $(e_i)$ is boundedly complete. Furthermore, $c_0$ and $\ell_1$ do not embed into $X$ if and only if $X$ is quasi-reflexive of order 1. It is interesting to note that these (except the skipped unconditionality result) were already observed in {the} 1970's by Brunel and Sucheston {\cite{BS}} for ESA { (equal sign additive) bases, which is a stronger property than spreading}. However, our results are more general and the proofs are different. The crucial part of our approach is an unconditionality result, Theorem { \ref{spreading main}}a, which is of independent interest. We also { show} that the well known averaging projection onto disjoint subsets of a subsymmetric basis remains bounded for the conditional spreading case as long as the subsets form a partition. One consequence is that $X$ is isomorphic to $D\oplus X$ where $D$ is the subspace spanned by $(d_{2n})_{n=1}^{\infty}$. Moreover, every Banach space with a spreading basis contains a complemented subspace with an unconditional basis. This answers an open problem of H. Rosenthal.

In Section 3 we make {a} few remarks on Banach spaces that admit conditional spreading models. Our study of the conditional spreading sequences were motivated by the problems discussed in this section.

In section 4 we consider spaces whose asymptotic structure is closely related to $c_0$. In \cite{OS} it was shown that if $(x_i)$ is a basis for $X$ and all spreading models of normalized block bases of $(x_i)$ are 1-equivalent to the unit vector basis of $c_0$, then $c_0$ embeds into $X$. Our first result of Section 4 generalizes this as follows. If $(x_i)$ is weakly null and if {every spreading model $(e_i)$ generated by a {\em weakly null} block basis satisfies $\|e_1-e_2\|=1$ and} $\ell_1\not\hookrightarrow X$, then $c_0\hookrightarrow X$. This yields a quick proof of the Elton-Odell theorem \cite{EO}. Namely, for every Banach space $X$ there exists an infinite sequence $(z_i)$ in the unit sphere $S_X$ and $\lambda>1$ so that $\|z_i-z_j\|\ge \lambda$ for all $i\neq j$. Indeed, if $X$ contains $\ell_1$ or $c_0$ the result follows easily by the non-distortability of $c_0$ and $\ell_1$. Otherwise, fix a weakly null normalized sequence $(x_i)$. By { our theorem, $(x_i)$ }must have {a normalized block basis with }a spreading model $(e_i)$ with $\|e_1-e_2\|>1$ which yields an $\ep>0$ and an infinite $(1+\ep)$-separated {sequence}. 

One of the long standing open problems on asymptotic structures of Banach spaces is the following. Suppose that every spreading model of $X$ is equivalent to the unit vector basis of $c_0$ (or $\ell_p$). Does $X$ contain an asymptotic-$c_0$ (or asymptotic-$\ell_p$) subspace? We solve the $c_0$ case with somewhat stronger assumption. If all normalized {\em asymptotic models} $(e_i)$ of normalized weakly null arrays in $X$ are equivalent to the unit vector basis of $c_0$ and $\ell_1\not\hookrightarrow X$, then $X^*$ is separable and $X$ is asymptotic-$c_0$ with respect to a shrinking basis $(y_i)$ of $Y\supseteq X$. Recall that $(e_i)$ { is }an asymptotic model of $X$, denoted by $(e_i)\in AM_w(X)$, if there exists a normalized array $(x^i_j)_{i, j\in\N}$ so that $(x^i_j)_{j=1}^{\infty}$ is weakly null for all $i\in\N$, and for some $\ep_n\downarrow 0$, all $n$ and all $(a_i)_1^n\subseteq [-1, 1]$ and $n\le k_1<k_2<\ldots <k_n$

\begin{equation}\label{as model}
\left| \Big\|\sum_{i=1}^n a_i x^i_{k_i}\Big\|- \Big\|\sum_{i=1}^n a_i e_i\Big\|\right|\le\ep_n.
\end{equation}
The notion of asymptotic models is a direct generalization of spreading models and it was introduced in \cite{HO}. $X$ is asymptotic-$c_0$ if for some $K<\infty$ for all $n$ and all asymptotic spaces $(e_i)_{i=1}^n$ are $K$-equivalent to the unit vector basis of $\ell_{\infty}^n$ \cite{MMT}. These notions are recalled in Section 4.

\section{Spreading bases}

We begin with a result solving a problem asked of us by S. A. Argyros.

\begin{thm}\label{skipped reflexive}
Let $(e_n)$ be a normalized basis for $X$. If every subspace spanned by a skipped block basis of $(e_n)$ is reflexive then $X$ is either reflexive or quasi-reflexive of order 1.
\end{thm}

\begin{proof}
The hypothesis yields that $(e_n)$ is shrinking. If not, then for some normalized block basis $(x_n)$ of $(e_n)$ there exists $f\in B_{X^*}$ and $\ep>0$ with $f(x_n)>\ep$ for all $n$. But then { $(x_{2n})$} is a skipped block basis of $(e_n)$ which cannot be shrinking, hence cannot span a reflexive space.

Let $F\in X^{**}$. Since the basis $(e_i)$ is shrinking $F$ is the $w^*$-limit of $(\sum_{i=1}^n F(e^*_i)e_i)_{n=1}^{\infty}$ where $(e^*_i)$ is the biorthogonal sequence to $(e_i)$ (a basis for $X^*$). We claim that if
$$\liminf_n |F(e^*_i)|=0,$$ then $F\in i(X)$, { where $i(X)$ is the natural embedding of $X$ into $X^{**}$.}

Indeed, pick a subsequence $(i_j)$ such that $\sum_{j=1}^{\infty}|F(e^*_{i_j})|<\infty$. Let $y=\sum_{j=1}^{\infty}F(e^*_{i_j})e_{i_j}$. Then $y\in i(X)$. Let $G=F-y$. Then $G=w^*-\lim_n\sum_{j=1}^n\sum_{i_j<i<i_{j+1}}F(e^*_i)e_i$ and $(\sum_{i_j<i<i_{j+1}}F(e^*_i)e_i)_{j=1}^{\infty}$ is a skipped block sequence which spans a reflexive subspace. Thus $G\in i(X)$ and so is $F$.

Now suppose $X$ is not reflexive and let $G\in X^{**}$ and $F\in X^{**}\setminus i(X)$.  Choose $\lambda\in\R$ and a subsequence $(i_n)$ of $\N$ so that $G(e^*_{i_n})-\lambda F(e^*_{i_n})\to 0$. Then by the claim above we conclude that $G-\lambda F\in i(X)$. Therefore $X^{**}=\R F\oplus i(X)$.
\end{proof}

\begin{remark}
A generalization of the above from a basis to finite dimensional decompositions (FDD) is false. Indeed, {the} Argyros-Haydon space $\mathfrak X_K$ has an FDD $(M_n)$ with the property that every skipped blocking of $(M_n)$ spans a reflexive subspace and yet its dual is isomorphic to $\ell_1$ (Theorem 9.1, \cite{AH}). We thank Pavlos Motakis for pointing out the example.
\end{remark}

%
%
%
%

We now turn to conditional spreading bases. Suppose that $(e_i)$ is a normalized spreading basis for $X$ which is not weakly null. Then the {\em  summing functional},
$$S(\sum_i a_i e_i):=\sum_i a_i$$
is bounded on $X$.
Indeed for some $\lambda\neq 0$, $f\in X^*${,} and subsequence $(i_n)$ of $\N$ {we have that} $f(e_{i_n})-\lambda\to 0$ rapidly. So a perturbation of $\lambda^{-1}f$ is constantly 1 on the $e_{i_n}$'s. Then it follows from the spreading property that $S$ is bounded on $X$.

By renorming we can assume that $(e_i)$ is normalized, 1-spreading and a bimonotone basis for $X$, and $\|S\|=1$. This is easily achieved by replacing $(e_i)$ by a spreading model of a subsequence, and then by the renorming $\vertiii{x}:=\max(\|x\|, |S(x)|)$. With this we also get that the functional $S_I(\sum_i a_i e_i):=\sum_{i\in I}a_i$ is of norm one for any interval $I$. Note that the boundedness of $S$ implies that the summing basis of $c_0$ is dominated by every conditional spreading sequence.

\begin{thm}\label{spreading main}
Let $(e_i)$ be a normalized 1-spreading, non weakly null, bimonotone basis for $X$. 
\begin{enumerate}
\item[a)] If $(x_i)$ is a normalized block basis of $(e_i)$ with $S(x_i)=0$ for all $i$, then $(x_i)$ is suppression 1-unconditional.
\item[b)] Let $(d_i)=(e_1, e_2-e_1, e_3-e_2, \ldots)$. Then $(d_i)$ is a skipped unconditional basis for $X$.
\item[c)] $(e_i)$ is boundedly complete if and only if $c_0\not\hookrightarrow X$.
\item[d)] $(d_i)$ is shrinking if and only if $\ell_1\not\hookrightarrow X$.
\item[e)] $\ell_1\not\hookrightarrow X$ if and only if $X^*=\R S\oplus [(e^*_i)]$.
\item[f)] $c_0$ and $\ell_1$ do not embed into $X$ if and only if $X$ is quasi-reflexive of order 1.
\end{enumerate}
\end{thm}

\begin{proof}
For { $x,y\in X$ which are finitely supported} with respect to the basis $(e_i)$, we write $x\sim y$ if
$$x=\sum_{i=1}^k a_i e_{n_i}\ \ {\rm and}\ \ y=\sum_{i=1}^k a_i e_{m_i}\ \ {\rm where}\ \ n_1<\ldots n_k,\ m_1<\ldots<m_k.$$

a) Let $(x_i)$ be as in a). { We now need the following lemma.}

\begin{lem}\label{unconditionality lemma}
For all $\ep>0$ and $i_0\in\N$ there exists $m\in\N$ such that for all $f\in S_{X^*}$ there exists $\tilde x \in X$, $\tilde x\sim x_{i_0}$ and ${\rm supp}(\tilde x)\subseteq [j, m]$, $j=\min{\rm supp}(x_{i_0})$, so that $|f(\tilde x)|<\ep$.
\end{lem}
\begin{proof}
Let $\ep>0$ and $i_0\in\N$. Since $|f(e_i)|\le 1$ for any $f\in S_{X^*}$, by the pigeonhole principle there exists $m$ with the following property:

Let $j=\min{\rm supp}(x_{i_0})$. For all $f\in S_{X^*}$ there exists $\lambda\in[-1,1]$ and $F\subseteq [j, m]$ with $|F|=k=|{\rm supp}(x_{i_0})|$ so that for $i\in F$, $|f(e_i)-\lambda |<\ep/k.$

Place $\tilde x\equiv \sum_{i\in F}a_i e_i$ on $F$ so that $\tilde x\sim x_{i_0}$. Then $S(\tilde x)=S(x_{i_0})=0$ and 
$$|f(\tilde x)|\le |f(\tilde x-\lambda S(\tilde x))|+|\lambda S(\tilde x)|=\Big|\sum_{i\in F} a_i (f(e_i)-\lambda)\Big|<\ep.$$
\end{proof}

Now let $x=\sum_{i=1}^k a_i x_i$, $\|x\|=1$, $\ep>0$. Let $F\subseteq\{1, 2, \ldots, k\}$. We will show that $\|\sum_{i\in F}a_i x_i\|\le 1+\ep.$ Let $j_i=\min{\rm supp}(x_i)$ for $i\le k$ and choose $m_i$ by Lemma \ref{unconditionality lemma} for $\ep/k$ and $j_i$. Since $(e_i)$ is 1-spreading we may assume that $j_1<m_1<j_2<m_2<\ldots$. Let $f\in S_{X^*}$ with $f(\sum_{i\in F}a_i x_i)=\|\sum_{i\in F}a_i x_i\|$. For $i\not\in F$, $i\le k$, we choose $\tilde x_i\sim x_i$ with ${\rm supp}(\tilde x_i)\subseteq [j_i, m_i]$ so that $|f(\tilde x_i)|<\ep/k$. Then
$$\left\|\sum_{i\in F}a_i x_i\right\|\le \left|f\Big(\sum_{i\in F}a_i x_i+ \sum_{i\not\in F}a_i \tilde x_i\Big)\right|+\left|f\Big(\sum_{i\not\in F}a_i \tilde x\Big)\right|\le { \|x\|+\ep= }1+\ep.$$
This proves a).

b) To see $(d_i)$ is a basis for $X$ we need only note that it is basic. This is an easy calculation that holds for any difference sequence $(d_i)$ obtained from a normalized basic $(e_i)$ that dominates the summing basis (i.e., $S$ is bounded). Indeed, for any $n<m$
\begin{eqnarray*}
\Big\|\sum_{i=1}^n a_i d_i\Big\|&=&\Big\|\sum_{i=1}^{n-1}(a_i-a_{i+1})e_i+a_n e_n\Big\|\le \Big\|\sum_{i=1}^m a_i d_i\Big\|+\|a_{n+1}e_n\|\\
&=&\Big\|\sum_{i=1}^m a_i d_i\Big\|+\Big|S\Big(\sum_{i=n+1}^{m-1} (a_i-a_{i+1})e_i+a_me_m\Big)\Big|\le 2 \Big\|\sum_{i=1}^m a_i d_i\Big\|.
\end{eqnarray*}

 That $(d_i)$ is skipped unconditional follows from a).

c) We need only show that if $(e_i)$ is not boundedly complete, then $c_0\hookrightarrow X$. Suppose that there exists $(a_i)\subseteq \R$ so that $\sup_n \|\sum_{i=1}^n a_i e_i\|=1$ and $\sum_{i=1}^{\infty}a_i e_i$ diverges. Choose $\delta>0$ and a subsequence $(k_i)$ of $\N$ so that $\|x_i\|>\delta$ where $x_i=\sum_{j=k_i}^{k_{i+1}-1}a_i e_i$ for $i\in\N$. 

Choose a block sequence $(y_i)$ of $(e_i)$ so that $y_{2i-1}\sim x_i$ and $y_{2i}\sim x_i$ for all $i$. Then $(y_{2i-1})$ and $(y_{2i})$ are each equivalent to $(x_i)$ and $(y_{2i-1}-y_{2i})$ is unconditional by a). Furthermore
$$\sup_n\|\sum_{i=1}^n(y_{2i-1}-y_{2i})\|\le 2$$
and $2\ge \|y_{2i-1}-y_{2i}\|\ge \delta$ for all $i$. Thus $(y_{2i-1}-y_{2i})$ is equivalent to the unit vector basis of $c_0$.

d) This follows easily since $(d_i)$ is skipped unconditional.

e) Suppose $\ell_1$ does not embed into $X$. By Rosenthal's $\ell_1$ theorem \cite{R} and the fact that $(e_i)$ is spreading, $(e_i)$ is weak Cauchy.

Let $f\in X^*$. Then $f=w^*-\lim_n \sum_{i=1}^n f(e_i)e^*_i$, and $\lim_{i\to \infty}f(e_i)\equiv \lambda$ exists. Then $f-\lambda S\in [(e^*_i)]$. Indeed $f-\lambda S=w^*-\lim_{n\to\infty}\sum_{i=1}^n b_ie^*_i$ where $\lim_i b_i=0$. If the series is not norm convergent there exists $\delta>0$, $(n_i)\in [\N]^{\omega}$, and a normalized block basis $(x_i)$ of $(e_i)$ so that $x_1<e_{n_1}<x_2<e_{n_2}<\ldots$, so that $(f-\lambda S)x_i>\delta$ for all $i$ and $b_{n_i}\to 0$ rapidly. In particular, $(x_i - S(x_i)e_{n_i})$ is unconditional and $(f-\lambda S)(x_i- S(x_i)e_{n_i})>\delta/2$ for all $i$. Thus $(x_i-S(x_i)e_{n_i})$ is equivalent to the unit vector basis of $\ell_1$, a contradiction.

f) Let $(u_n)$ be a skipped block basis of $(d_i)$, and assume $c_0$ and $\ell_1$ do not embed into $X$. Then $(u_n)$ is unconditional and shrinking by b) and d) and is also boundedly complete since $X$ does not contain $c_0$. Thus $[(u_n)]$ is reflexive and Theorem \ref{skipped reflexive} yields the result.
\end{proof}

If $X$ has { an}  unconditional basis and $Y\subseteq X$ has non separable dual then $\ell_1\embed Y$ \cite{BP}. This also holds if $X$ has a spreading basis. In fact, the result holds more generally.

\begin{prop}\label{non separable dual}
Suppose $X$ { has} a skipped unconditional basis and let $Y\subseteq X$ with $Y^*$ not separable. Then $\ell_1$ embeds into $Y$.
\end{prop}

\begin{proof}
Assume that $Y^*$ is not separable and $\ell_1$ does not embed into $Y$. By Theorem 3.14 of \cite{AJO} there exists an $\ell^+_1$ weakly null tree $(y_{\alpha})_{\alpha\in T_{\omega}}$ in $Y$. Here $T_{\omega}=\{(n_i)_1^k: n_1<\ldots<n_k, n_i\in\N, k\in\N\}$. $(y_{(\alpha, n)})_n$ is weakly null and normalized for all $\alpha\in \{\emptyset\}\cup T_{\omega}$. Furthermore, for some $c>0$, $\|\sum_i a_i y_{\alpha_i}\|\ge c\sum_i a_i$ for all branches $(\alpha_i)$ of $T_{\omega}$ and $a_i\ge 0$. Using that the tree is weakly null and $X$ has a skipped unconditional basis it is easy to find a branch $(y_{\alpha_i})$ which is unconditional, hence is equivalent to the unit vector basis of $\ell_1$. This is a contradiction.
\end{proof}

%

\begin{remark}
The same proof also yields that if $X$ is a subspace of a space with skipped { unconditional finite dimensional decomposition} and $X^*$ is non-separable, then $\ell_1$ embeds into $X$.
\end{remark}

The next result answers a question asked of us by H. Rosenthal: If $X$ has a spreading basis, does $X$ contain a complemented subspace with an unconditional basis?

\begin{prop}\label{complemented unconditional subspace}
If $(e_i)$ is a normalized spreading basis for $X$ then the subspace $Y$ spanned by {the unconditional block} basis $[(e_{2n-1}-e_{2n})]$ is complemented in $X$.
\end{prop}

It suffices to prove that the complementary ``projection" $Q$ is bounded where
$$Q(\sum_i a_i e_i)=\sum_i \frac{a_{2i-1}+a_{2i}}{2}(e_{2i-1}+e_{2i}).$$

This is a consequence of the following more general result which is well known if the basis is subsymmetric.

\begin{thm}\label{averaging projection}
Let $(e_i)$ be a normalized bimonotone 1-spreading basis for $X$. Let $(\sigma_j)_{j=1}^{\infty}$ be a partition of $\N$ into successive intervals, $\sigma_1<\sigma_2<\ldots$, with $|\sigma_j|=n_j$ for $j\in\N$. Then the averaging operator
$$Q\Big(\sum_i a_i e_i\Big)=\sum_{j=1}^{\infty}\Big(\big(\sum_{i\in\sigma_j}a_i\big)/n_j\Big)\Big(\sum_{i\in\sigma_j}e_i\Big)$$
is a bounded projection on $X$ with $\|Q\|\le 3$.
\end{thm}

It is important to note that, unlike the subsymmetric case, there are no gaps allowed between blocks in this averaging operator.

\begin{proof}
It suffices to prove that for all $k$, $\|Qx\|\le 3\|x\|$ if ${\rm supp}(x)\subseteq \bigcup_{i=1}^k\sigma_i$. Let $k\in\N$, $x=\sum_{j=1}^{\max(\sigma_k)}a_je_j$. Let $M$ be the least common multiple of $(n_1, n_2, \ldots, n_k)$ and set $m_j=M/n_j$ for $j\le k$. 

We will construct vectors $(y_i)_{i=1}^{2M}$ so that $\frac{1}{2M}\sum_{j=1}^{2M}y_j=\bar x+\sum_{j=1}^M z_j$ where $y_i\sim x$, $2\bar x\sim Qx$ and $z_j\sim\frac{1}{2M}x$ for $j\le M$. It follows that
$$\|Qx\|=2\|\bar x\|\le { 2}\Big(\|x\|+M\frac{1}{2M}\|x\|\Big){ =} 3\|x\|.$$

To begin we spread $x$ to obtain $y_1$ so that the coordinates of $y_1$ looks like this
$$y_1=(a_1, a_2, \ldots, a_{n_1}, 0, \ldots, 0, a_{n_1+1}, \ldots, a_{n_1+n_2}, 0,\ldots, 0, a_{n_1+n_2+1}, \ldots).$$

{  For each $1\leq j\leq k-1$, we insert $2n_j-1$ zeros between the blocks of $x$ corresponding to $\sigma_j$ and $\sigma_{j+1}$, and let
 $\gamma_j$ be the index set for the coordinates of the inserted block of zeros.
The vectors $y_2, \ldots, y_{2M}$ will be  spreads of $y_1$.} The position of the first block $(a_1, \ldots, a_{n_1})$ is preserved for $y_2,\ldots, y_{m_1}$. This block is then shifted one unit right for $y_{m_1+1}, \ldots, y_{2m_1}$. Then another unit to the right for $y_{2m_1+1}, \ldots, y_{3m_1}$ and so on { $n_1$ times until reaching $y_{2M}=y_{2n_1 m_1}$}. The same scheme is followed for the second block $(a_{n+1}, \ldots, a_{n_1+n_2})$ and the subsequent blocks. Thus the second block is preserved for $y_2,\ldots, y_{m_2}$ and then shifted once right for $y_{m_2+1},\ldots, y_{2m_2}$.

When we average the $y_j$'s, $\bar x$ will be the average { of the vectors $y_1,...,y_{2M}$ restricted to the coordinates given by the union over $1\le j\le k$ of the first $n_j$ coordinates of $\gamma_j$.}

We give a simple example in the diagram below explaining this averaging procedure in the case $k=2$, $n_1=2$, $n_2=3$ and so $M=6$, $m_1=3$, and $m_2=2.${ 

\[
\begin{array}{cccccccccccccc}
a_1 & {a_2} & 0 & 0 & 0  & a_3 & a_4 & {a_5} & 0 & 0 & 0 & 0 & 0\\
a_1 & {a_2} & 0 & 0 & 0  & a_3 & a_4 & {a_5} & 0 & 0 & 0 & 0 & 0\\
a_1 & {a_2} & 0 & 0 & 0  & 0 & a_3 & {a_4} & {\bf a_5} & 0 & 0 & 0 & 0\\
0 & { a_1} & {\bf a_2} & 0 & 0 & 0 & a_3 & {a_4} & {\bf a_5} & 0 & 0 & 0 & 0\\
0 & { a_1} & {\bf a_2} & 0  & 0 & 0 & 0 & { a_3} & {\bf a_4} & {\bf a_5} & 0 & 0 & 0\\
0 & { a_1} & {\bf a_2} & 0  & 0 & 0 & 0 & { a_3} & {\bf a_4} & {\bf a_5} & 0 & 0 & 0\\
0 & 0 & {\bf a_1} & {\bf a_2} & 0 & 0 & 0 & 0 & {\bf a_3} & {\bf a_4} & {\bf a_5} & 0 & 0\\ 
0 & 0 & {\bf a_1} & {\bf a_2} & 0 & 0 & 0 & 0 & {\bf a_3} & {\bf a_4} & {\bf a_5} & 0 & 0\\ 
0 & 0 & {\bf a_1} & {\bf a_2} & 0 & 0 & 0 & 0 & 0 & {\bf a_3} & {\bf a_4} & a_5 & 0\\
0 & 0 & 0 & {\bf a_1} & { a_2} & 0 & 0 & 0 & 0 & {\bf a_3} & {\bf a_4} & a_5 & 0\\
0 & 0 & 0 & {\bf a_1} & { a_2} & 0 & 0 & 0 & 0 & 0 &  {\bf a_3} & { a_4} & a_5\\
0 & 0 & 0 & {\bf a_1} & { a_2} & 0 & 0 & 0 & 0 & 0 &  {\bf a_3} & a_4 & a_5 \\ 
\\
\end{array}
\]

The vector $\bar x$ is the average of $y_1,...,y_{2M}$ restricted to the coordinates given by bold type.  The remaining coefficients are easily partitioned into $M$ spreads of $x$.}
\end{proof}

\begin{prop}\label{decomposition}
Let $(e_i)$ be a normalized conditional spreading basis for $X$. Let $D=[(d_{2n})]$, where $(d_n)$ is the difference basis. Then $X\simeq D\oplus Y$ where $Y=[(e_1+e_2, e_3+e_4, \ldots)]$ is isomorphic to $X$.
\end{prop}

\begin{proof}
We may assume $(e_i)$ is 1-spreading. By Proposition \ref{complemented unconditional subspace} and Theorem \ref{averaging projection} it suffices to prove that $(e_{2n-1}+e_{2n})_{n=1}^{\infty}$  { dominates} $(e_n)$. { We will} prove that if $x=\sum_{i=1}^n a_i e_i$, $\|x\|=1$, then $\|\sum_{i=1}^n a_i (e_{2i-1}+e_{2i})\|\ge 2/3$. Write $x_1=\sum_{i=1}^n a_i e_{3i-1}$, $x_2=\sum_{i=1}^n a_i e_{3i-2}$ and $x_3=\sum_{i=1}^n a_i e_{3i}$. Assume $\|x_1+x_2\|=c$. Let $f\in S_{X^*}$, $1=f(x_1)$. Then $f(x_1+x_2)\le c$ so $f(x_2)\le c-1$. Also using $\|x_1+x_3\|=c$, $f(x_3)\le c-1$. Thus $c\ge -f(x_2+x_3)\ge 2-2c$ and so $c\ge 2/3$. 
{ Thus, $\|\sum_{i=1}^n a_i (e_{2i-1}+e_{2i})\|=\|\sum_{i=1}^n a_i (e_{3i-2}+e_{3i-1})\|=c\ge2/3$.}
Note that the argument can easily be generalized for all $\ep>0$ to get $c\ge 1-\ep$. 
\end{proof}

It has been shown that spaces $X$ whose dual are isomorphic to $\ell_1$ are quite plentiful and need not contain $c_0$ \cite{BD}. Moreover, any $Y$ with $Y^*$ separable embeds into such a space \cite{FOS}. But if $X$ has a spreading basis, $X^*$ is separable and $\ell_1\embed X^*$, then $c_0\embed X$. This holds more generally if $X^*$ is separable and $X^{**}$ is not separable, assuming a spreading basis, by Theorem \ref{spreading main}.  More can be said if $X^*$ is isomorphic to $\ell_1$.

\begin{thm}\label{spreading l_1 dual}
Let $(e_i)$ be a normalized spreading basis for $X$ and assume $X^*$ is isomorphic to $\ell_1$. Then $(e_i)$ is equivalent to either the unit vector basis of $c_0$ or the summing basis.
\end{thm}

\begin{proof}
If $(e_i)$ is weakly null then, it is unconditional. It follows that $(e^*_i)$ is subsymmetric. Since $X^*\simeq \ell_1$ some subsequence of $(e^*_i)$ is equivalent to the unit vector basis of $\ell_1$, so $(e^*_i)$ is such and so $(e_i)$ is equivalent to the unit vector basis of $c_0$.

If $(e_i)$ is not weakly null,  then we consider the difference basis $(d_i)$ of $X$. To show $(e_i)$ is equivalent to the summing basis it suffices to show that $(d_i)$ is equivalent to the unit vector basis of $c_0$. To do this, it suffices, by the triangle inequality, to show $(d_{2i})$ is equivalent to the unit vector basis of $c_0$ since $(d_{2i})$ is equivalent to $(d_{2i-1})$. Now $D=[(d_{2n})]$ is complemented in $X$ and $(d_{2n})$ is unconditional and shrinking. So $(d^*_{2n}|_D)$ is an unconditional basis for $D^*$ which is isomorphic to $\ell_1$, since it is complemented in $X^*\simeq \ell_1$. Thus $(d^*_{2n}|_D)$ is equivalent to the unit vector basis of $\ell_1$. These are due to the facts that $\ell_1$ is prime and has unique unconditional basis. {Hence} $(d_{2n})$ is equivalent to the unit vector basis of $c_0$.
\end{proof}

\section{Remarks on conditional spreading models}
Recall that a normalized basic sequence $(e_i)$ is a spreading model of a sequence $(x_i)$ if for some $\ep_n\downarrow 0$, for all $n$, $(a_i)_1^n\subseteq [-1, 1]$ positive integers $n\le k_1<\ldots<k_n$

\begin{equation}\label{sp model}
\left | \Big\|\sum_{j=1}^n a_j x_{k_j}\Big\|-\Big\|\sum_{i=1}^n a_i e_i\Big\|\right |\le \ep_n.
\end{equation}

In this case $(e_i)$ is 1-spreading, and 
if $(x_i)$ is weakly null, then $(e_i)$ is suppression 1-unconditional.  We denote by $SP_w(X)$ the set of all spreading models of $X$ generated by weakly null sequences. If $(y_i)$ is normalized basic then, via Ramsey theory, some subsequence $(x_i)$ of $(y_i)$ generates a spreading model $(e_i)$ as in (\ref{sp model}) above. 
{ If $(y_i)$ is normalized but does not have a basic subsequence then any basic spreading model admitted by $(y_i)$ must be} equivalent to the unit vector basis of $\ell_1$.  Indeed,  by Rosenthal's $\ell_1$ theorem we may assume $(y_i)$ is weak Cauchy. { Every} non-trivial weak Cauchy sequence { has } a basic subsequence (see the proof of Proposition 2.2, \cite{R1}). Thus a subsequence $(x_i)$ of $(y_i)$ weakly converges to a nonzero element $x_0$, and $(x_i-x_0)$ generates an unconditional spreading model $(u_i)$. So $(e_i)$ is equivalent to $(x_0+u_i)$ in $\langle x_0\rangle\oplus [(u_i)]$. Since $(e_i)$ is basic, $(u_i)$ is not weakly null and therefore equivalent to the unit vector basis of $\ell_1$, and so is $(e_i)$. 


One of the questions of interest about spreading models is whether there exists a ``small'' space that is universal for all (or a large class of) spreading models. Recall that the space $C(\omega^{\omega})$ is universal for all unconditional spreading models, that is, every subsymmetric basic sequence is a spreading model of $C(\omega^{\omega})$ \cite{O}.  In \cite{AM} a remarkable example of a {\em reflexive} space is constructed so that {\em every infinite dimensional subspace} of it is universal for all unconditional spreading models. For the case of conditional spreading models, S. A. Argyros raised the following which partly motivated our study of conditional spreading sequences above.
\begin{prob}\label{Argyros' question}
Let $(e_i)$ be a conditional normalized spreading sequence. Does there exists a quasi-reflexive of order 1 space $X$ with a normalized basis $(x_i)$ which generates $(e_i)$ as a spreading model? 
\end{prob} 

{ We show that the answer is affirmative for the summing basis of $c_0$.} For a given basis $(e_i)$, recall the space $J(e_i)$. For $x\in J(e_i)$, the norm is given by
$$\|x\|=\sup\Big\{\Big\|\sum_{i=1}^k s_i(x)e_{p_i}\Big\|: s_1<s_2<\ldots<s_k\ {\rm { are\ } intervals\ in\ }\N,\ \min s_i=p_i\Big\},$$ where $s_i(x)=\sum_{j\in s_i}a_j$, $s_i=[p_i, q_i)$, and $x=(a_j)$. 
\begin{prop}\label{summing basis}
Let $(e_i)$ be the unit vector basis of the dual Tsirelson space $T^*$. Then the space $J(e_i)$ is quasi-reflexive of order 1 and the spreading model generated by its natural basis is equivalent to the summing basis { of $c_0$}.
\end{prop}

\begin{proof}
In \cite{BHO} it is shown that if $(e_i)$ is a basis of a reflexive space, then $J(e_i)$ is quasi-reflexive of order 1. Thus the first assertion follows since $T^*$ is reflexive.

Also it is easy to see that any subsequence of the basis $(u_i)$ of $J(e_i)$ generates a spreading model equivalent to the summing basis $(s_i)$. Indeed, to estimate the norm of a vector $x=\sum_{j=1}^k a_j u_{i_j}$ where $k\le i_1<\ldots<i_k$ note that for an arbitrary $s_1<\ldots<s_k$ we have

$$\Big\|\sum_{j=1}^k s_j(x)e_{i_j}\Big\|_{T^*}\le 2\max_{j}|\sum_{i\in s_j}a_i|$$ and the latter expression is at most twice the summing norm of $x$. The reverse inequality is trivial (consider intervals $s=[l, i_k],  k\le l\le i_k$).
\end{proof}

In a follow-up work \cite{AMS} the constructions similar to the above are studied in more detail and, in particular, Problem \ref{Argyros' question} is solved affirmatively.

\section{Spreading and asymptotic models}

Our first result of { this} section is a strengthening of the $c_0$-part of the following theorem of Odell and Schlumprecht \cite{OS}. {\em If $X$ has a basis $(x_i)$ so that every spreading model of a normalized block basis of $(x_i)$ is 1-equivalent to the unit vector basis of $c_0$ (respectively, $\ell_1$), then $X$ contains an isomorphic copy of $c_0$ (respectively, $\ell_1$).}  Here we show that it is sufficient to restrict the assumption to those spreading models generated by weakly null block bases.

\begin{thm}\label{c_0 spreading models}
Let $(x_i)$ be a normalized weakly null basis for $X$. Assume that $\ell_1$ does not embed into $X$ and whenever $(y_i)$ is a normalized weakly null block basis of $(x_i)$ with spreading model $(e_i)$, then $\|e_1-e_2\|=1$. Then some subsequence of $(x_i)$ is equivalent to the unit vector basis of $c_0$.
\end{thm}

\begin{rem}
The hypothesis yields that every spreading model $(e_i)$ generated by a weakly null normalized sequence $(y_i)$ is 1-equivalent to the unit vector basis of $c_0$. Indeed, we may assume $(y_i)$ is a weakly null normalized block basis of $(x_i)$. Then $\left(\frac{y_{2n-1}-y_{2n}}{\|y_{2n-1}-y_{2n}}\| \right)$ is a weakly null block basis generating the normalized spreading model $(e_{2n-1}-e_{2n})$ and so $\|e_1-e_2-e_3+e_4\|=1$. By iteration of this argument, 1-spreading and the suppression 1-unconditionality of $(e_i)$, 
$$\big\|\sum_{i=1}^n \pm e_i\big\|=1\ \text{for all}\ \pm 1\ \text{and all}\ n.$$
This implies $(e_i)$ is 1-equivalent to the unit vector basis of $c_0$.
\end{rem}

As it was pointed out in { the introduction} this immediately implies  { the }following well known theorem of Elton and Odell \cite{EO}.

\begin{thm}[Elton-Odell]\label{Elton-Odell} Let $X$ be an infinite dimensional Banach space. Then there exists $\lambda>1$ and an infinite sequence $(x_i)\subset S_X$ such that $\|x_i-x_j\|\ge \lambda$ for all $i\neq j$.
\end{thm}

For the proof of Theorem \ref{c_0 spreading models} we need to recall some terminology. A collection $\F\subseteq [\N]^{<\omega}$ is called {\em thin} if there do not exist $F, G\in\F$ with $F$ being a proper initial segment of $G$. $\F$ is {\em large} in $M\in[\N]^{\omega}$ if for all $N\in [M]^{\omega}$ there exists an initial segment $F$ of $N$ with $F\in \F$. For a sequence $(x_i)\subseteq X$ and $E\in [\N]^{<\omega}$ we set $x_E=\sum_{i\in E}x_i$. For a thin $\F\subseteq [\N]^{<\omega}$ we let 
$$\F^I=\{G{\in[\N]^{<\omega}}: G\  \text{is an initial segment of some } F\in \F\}.$$

\begin{lem}\label{weakly null}
Let $X$ and $(x_i)$ be as in the hypothesis of Theorem \ref{c_0 spreading models}. Let $\F$ be a collection of finite subsets of $\N$ satisfying 
\begin{equation}\label{bounded thin}
\sup \{\|x_E\|: E\in \F\}<\infty.
\end{equation}
Then there exists $M\in [\N]^{\omega}$ so that for all $E_1<E_2<\ldots ${ with $E_i\in \F\cap[M]^{<\omega}$ for all $i\in\N$, the sequence }$(x_{E_i})$ is weakly null.
\end{lem}

\begin{proof}
By Elton's near unconditionality theorem \cite{E}, there exists $M\subseteq \N$ such that for some $C<\infty$ the subsequence $(x_i)_{i\in M}$ satisfies { for all  $E\subseteq F\in [M]^{<\omega}$,
\begin{equation}\label{Elton}
 \big\|\sum_{i\in E}\delta_i x_i\big\|\le C \big\|\sum_{i\in F}\delta_i x_i\big\|
\quad\text{for all choices of signs}, \ \delta_i=\pm 1.
\end{equation}
}

Suppose that for some $E_1<E_2<\ldots $, $E_i\in \F$ { with } $E_i\subseteq M$ for all $i$, { the sequence }
$(x_{E_i})$ is not weakly null. Then { after passing to a subsequence, there exists $\ep>0$ and $f\in B_{X^*}$  so that $f(x_{E_j})>\ep$ for all $j\in \N$.} Since $X$ does not contain $\ell_1$, by Rosenthal's $\ell_1$ theorem and passing to a { further} subsequence, we may assume that $(x_{E_j})$ is weak Cauchy.

Let $z_j=x_{E_{2j-1}}-x_{E_{2j}}$ for $j\in \N$. Then $(z_j)$ is weakly null and moreover by (\ref{Elton})
\begin{equation*}
n \ep\le \big\|\sum_{j\in G}x_{E_{2j-1}}\big\|\le C\big\|\sum_{j\in G} z_j\big\|
\end{equation*}
for all $|G|=n$, $n\in \N$. Thus $(z_j/\|z_j\|)_j$ cannot have a $c_0$ spreading model since $\sup_j \|z_j\|<\infty$ by the assumption (\ref{bounded thin}).
\end{proof}

\begin{lem}\label{induction step}
Let $X$ and $(x_i)$ be as in the hypothesis of Theorem \ref{c_0 spreading models}. Let $\F$ be a thin { collection of finite subsets of $\N$ which is large in $\N$}. Assume that $(x_{E_i})$ is weakly null for all $E_1<E_2<\cdots$ in $\F$ and 
\begin{equation}\label{boundedness assumption}
\limsup_n \{\|x_E\|:E\in \F, n\le E\}=1.
\end{equation}
Then there exists $N=(n_i)\in [\N]^{\omega}$ so that $\G$, defined by,
\begin{equation*}
\G=\Big\{\bigcup_{i=1}^k E_i:k\in\N, n_k=\min(E_1), \
E_1<\cdots<E_k,  { E_i\in\F\cap[N]^{<\omega}}\ {\rm for}\ i\le k\Big\}
\end{equation*}
is thin and large in $N$ and furthermore $\G$ satisfies (\ref{boundedness assumption}) (when $\G$ replaces $\F$).
\end{lem}

\begin{proof}
First we note that by passing to a subsequence, using $(x_i)$ is normalized and weakly null, we may assume that {
\begin{equation}\label{E:liminf}\liminf_{n\to\infty}\{\|x_{E}\|: n\le E\in [\N]^{<\omega}\}\ge 1.\end{equation}
Indeed, for each $j\in\N$ we may choose $f_j\in X^*$ with $\|f_j\|=1$ such that $f_j(x_j)=\|x_j\|=1$.  Fix  $\delta_n\downarrow 0$ and after passing to a subsequence we may assume that $f_n(x_j)<\delta_n 2^{-j} $ for each $n<j$.  Thus, $1-\delta_n< f_{\min E}(x_E)\leq\|x_E\|$, for all $n\le E$ and \eqref{E:liminf} follows.
 }


Let $\ep_k\downarrow 0$ and set 
\begin{multline*}
\A_k=\Big\{ M\in [\N]^{\omega}: {\rm if}\ E_1<\cdots<E_k,\ E_i\in\F\ {\rm for}\ i\le k,\\
E=\bigcup_{i=1}^k E_i\ \text{is an initial segment of}\ M, \ {\rm then}\ \|x_E\|\le 1+\ep_k\Big\}.
\end{multline*}
 { Note that as $\F$ is thin and large in $\N$, for each $M\in[\N]^\omega$ there exists unique $E_1<\cdots<E_k$ with $\ E_i\in\F$ for $1\leq i\leq k$ such that $\cup_{i=1}^k E_i$ is an initial segment of $M$.  Thus,
whether or not a sequence $M\in [\N]^{\omega}$ is contained in $A_k$ depends entirely on a unique initial segment of $M$.  This makes  $A_k\subset [\N]^\omega$ open in the product topology.  Open sets are Ramsey, }
 so we can find subsequences of $\N$, $M_1\supset M_2\supset\cdots$, so that either $[M_k]^{\omega}\subseteq \A_k$ or $[M_k]^{\omega}\cap \A_k=\emptyset$ for each $k$.

By the 1-equivalent to $c_0$ spreading model hypothesis we must always have $[M_k]^{\omega}\subseteq \A_k$. Let $N=(n_i)$ be a diagonal sequence, $(n_i)_{i=k}^{\infty}\in M_k$ for all $k$. Define $\G$ as in the statement of the lemma with respect to $N$. 
\end{proof}

\begin{proof}[Proof of Theorem \ref{c_0 spreading models}]
We may assume, using \cite{E} as in the proof of Lemma \ref{weakly null}, that for some $C<\infty$,
\begin{equation}\label{Elton2}
\|x_E\|\le C\|x_F\|\ \text{for all}\ E\subseteq F\in [\N]^{{<\omega}}.
\end{equation}

We will show that for $\alpha<\omega_1$ there exists $N_{\alpha}=(n^{\alpha}_i)_i\in[\N]^{\omega}$ and $\G_{\alpha}\subseteq [N_{\alpha}]^{<\omega}$ so that $\G_{\alpha}$ is thin and large in $N_{\alpha}$. Moreover, {$\G^I_{\alpha}$ has} Cantor-Bendixson index $CB(\G^I_{\alpha})\ge \omega^{\alpha}$ and 
\begin{equation}\label{index}
\sup\{\|x_E\|: E\in \G_{\alpha},\ n^{\alpha}_k\le E\}\le 1+\ep_k
\end{equation}
where $\ep_k\downarrow 0$ is fixed.  { By \eqref{Elton2} we have that 
\begin{equation}\label{index2}
\sup\{\|x_E\|: E\in \G^I_{\alpha},\ n^{\alpha}_k\le E\}\le C(1+\ep_k).
\end{equation}

 Recall that if $K$ is a countable set then its Cantor-Bendixson index will be a countable ordinal.  Thus, the Cantor-Bendixson index of $\cup_{\alpha<\omega_1} G^I_\alpha$ is uncountable and }it follows that for some $N=(n_i)\in[\N]^{\omega}$, $1_N$ is in the pointwise closure of $$\{1_E: \|x_E\|\le 2C,\ E\in[\N]^{<\omega}\}{\textrm{ in }\{0,1\}^\N}.$$ Thus $\sup_k\|\sum_{i=1}^k x_{n_i}\|<\infty$ and by (\ref{Elton2}) we obtain that $(x_{n_i})$ is equivalent to the unit vector basis of $c_0$.

To begin we use Lemma \ref{induction step} applied to $\{\{j\}: j\in \N\}$ to obtain $N_1=(n^1_i)$ and $\G_1=\{E: n^1_k=\min E,\ |E|=k,\ E\subseteq N_1\}$ satisfying (\ref{index}), and note that $CB(\G^I_1)=\omega$. Assume $N_{\alpha}$ and $\G_{\alpha}$ are chosen to satisfy the given conditions. Choose $\tilde N_{\alpha+1}\subseteq N_{\alpha}$ by Lemma \ref{weakly null}. Then apply Lemma \ref{induction step} to $\tilde N_{\alpha+1}$ and $\G_{\alpha}$ to obtain $N_{\alpha+1}$ and $\G_{\alpha+1}$. By the definition of $\G_{\alpha+1}$, $CB(\G^I_{\alpha+1})\ge \omega^{\alpha+1}$.

If $\alpha$ is a limit ordinal, choose $\beta_n\uparrow \alpha$, and let $\tilde N_{\alpha}$ be a diagonal sequence of $(N_{\beta_n})$ so that $(\tilde n^{\alpha}_i)_{i=k}^{\infty}\subseteq N_{\beta_k}$ and (\ref{index}) holds. Let $\tilde\G_{\alpha}=\{E\subseteq \tilde N_{\alpha}: E\subseteq \G_{\beta_n}\ \text{for some}\ n\}$. Apply Lemmas \ref{weakly null} and \ref{induction step} as above.
\end{proof}

Recall that { the \em $n$-dimensional asymptotic structure of $X$} (with respect to a fixed filter $\text{cof}(X)$ of finite co-dimensional subspaces of $X$) is the collection $\{X\}_n$ of
normalized basic sequences $(e_i)_1^n$ satisfying the following.  For all $\ep>0$ and all $X_1\in \text{cof}(X)$ there exists $x_1\in S_{X_1}$ such that for all $X_2\in \text{cof}(X)$ there exists $x_2\in S_{X_2}$ so on so that for all $X_n\in \text{cof}(X)$ there exists $x_n\in S_{X_n}$ so that $(x_i)_1^n$ is $(1+\ep)$-equivalent to $(e_i)_1^n$ \cite{MMT}.  $X$ is asymptotic-$c_0$ if for some $K<\infty$ and all $n$, $(e_i)_1^n\in \{X\}_n$ implies that $(e_i)_1^n$ is $K$-equivalent to the unit vector basis of $\ell_{\infty}^n$. In this case $X^*$ must be separable and the condition can be described in terms of weakly null trees. Namely, $X$ is asymptotic-$c_0$ (assuming $X^*$ is separable) if and only if for some $K<\infty$ for all $n\in\N$ and all normalized weakly null trees $(x_{\alpha})_{\alpha\in T_n}$ in $X$, some branch is $K$-equivalent to the unit vector basis of $\ell_{\infty}^n$ where
$T_n=\{(k_1, k_2, \ldots, k_n): 1\le k_1<\cdots<k_n\}$.

The following question is open.

\begin{prob}\label{as c_0 problem}
Suppose that $\ell_1$ does not embed into $X$ and every spreading model generated by weakly null normalized sequences in $X$ is equivalent to the unit vector basis of $c_0$. Does $X$ contain an asymptotic-$c_0$ subspace? Does $X$ contain a subspace $Y$ with $Y^*$ separable?
\end{prob}

%

Note that the space $JH$ constructed by Hagler \cite{H} has non separable dual, does not contain $\ell_1$ and every weakly null normalized sequence has a subsequence equivalent to the unit vector basis of $c_0$. So if the problem has affirmative answer it is necessary to pass to a subspace.
We will prove a weaker theorem.

\begin{thm}\label{asymptotic c_0}
Suppose that a Banach space $X$ does not contain an isomorphic copy of $\ell_1$ and every asymptotic model $(e_i)$ generated by weakly null arrays in $X$ is equivalent to the unit vector basis of $c_0$. Then 

i) $X^*$ is separable and thus $X$ embeds into a space $Y$ with a shrinking basis $(y_i)$.

ii) $X$ is asymptotic-$c_0$ (with respect to the basis $(y_i)$). 
\end{thm}

Recall that $(e_i)$ is an { \em asymptotic model of $X$}, denoted by $(e_i)\in AM_w(X)$, generated by a normalized weakly null array $(x^i_j)_{i, j\in\N}$ if $(x^i_j)_{j=1}^{\infty}$ is weakly null for all $i\in\N$, and for some $\ep_n\downarrow 0$, all $n$ and all $(a_i)_1^n\subseteq [-1, 1]$ and $n\le k_1<k_2<\ldots <k_n$

\begin{equation}\label{as model}
\left| \Big\|\sum_{i=1}^n a_i x^i_{k_i}\Big\|- \Big\|\sum_{i=1}^n a_i e_i\Big\|\right|\le\ep_n.
\end{equation}
Asymptotic models were introduced in \cite{HO}.  { If every $(e_i)\in AM_w(X)$ is equivalent to the unit vector basis of $c_0$, then there exists} $K<\infty$ so that every $(e_i)\in AM_w(X)$ is $K$-equivalent to the unit vector basis of $c_0$ \cite{HO}.

The hypothesis of the theorem can be contrasted with being asymptotic-$c_0$ as follows. The asymptotic model condition implies that for some $K$, and every $n\in \N$ and normalized weakly null tree $(x_{\alpha})_{\alpha\in T_n}$ of a certain type, some branch is $K$-equivalent to the unit vector basis of $\ell^n_{\infty}$. The ''certain type'' condition is: there exist $n$ normalized weakly null sequences $(x^i_j)_{j=1}^{\infty}$, $1\le i\le n$ so that if $\alpha=(\ell_1, \ldots, \ell_k)$ then $x^k_{\ell_k}=x_{\alpha}$.  In short, the successor sequences to each $|\beta|=k-1$ are all the same, depending only on $k$, for all $1\le k\le n$.  { Theorem \ref{asymptotic c_0} states that if these specific normalized weakly null trees in $X$ each have a branch $K$-equivalent to the unit vector basis of $\ell_\infty^n$ then all normalized weakly null trees $(x_{\alpha})_{\alpha\in T_n}$ in $X$ do as well.}

\begin{proof}[Proof of Theorem]
i) We first show that $X^*$ is separable. Assume not. By a result of Stegall \cite{S} for all $\ep>0$ there exists $\Delta\subseteq S_{X^*}$, $\Delta$ is $w^*$-homeomorphic to the Cantor set, and a Haar like system $(x_{n,i})\subseteq X$. More precisely, there exist a sequence $(A_{n,i})$ of subsets of $\Delta$ for $n=0, 1, 2, \ldots$ and $i=0, 1, \ldots, 2^n -1$ such that $A_{0,0}=\Delta$ and each $A_{n,i}$ is the union of disjoint, non-empty, clopen subsets $A_{n+1, 2i}$ and $A_{n+1, 2i+1}$ with $\lim_{n\to \infty}\sup_{0\le i<2^n}\text{diam}(A_{n,i})=0$, and Haar functions $h_{n,i}\subseteq C(\Delta)$ (relative to $(A_{n,i})$) so that 
$$h_{2^n+i}:=1_{A_{n+1, 2i}}-1_{A_{n+1, 2i+1}},\ \ n=0, 1, \ldots, \ i=0, 1,\ldots, 2^n-1.$$
Finally,  $(x_{n,i})\subseteq X$ is a Haar like system (relative to $(A_{n,i})$) if, indexing above Haar functions as $h_{2^n+i}=h_{n,i}$, we have
$\|x_{n,i}\|\le 1+\ep$ for all $(n,i)$ so that
$$\sum_{n=0}^{\infty}\sum_{i=0}^{2^n-1}\|x_{n,i}|_{\Delta}-h_{n,i}\|_{C(\Delta)}<\ep.$$

For simplicity in what follows we will assume $x_{n,i}|_{\Delta}=h_{n,i}$ and ignore the tiny perturbations, and we will refer to the sets $A_{n,i}$'s as intervals. We will construct a Rademacher type system $(r_n)$ from the $x_{n,i}$'s and conclude that $\ell_1\embed X$ to get a contradiction. 

Begin with $r_1\equiv x_{0,0}$ and suppose $r_1, \ldots, r_n\in \text{span}(x_{k,i})$ have been constructed so that for each choice of signs $(\ep_i)_1^n$ there is an interval $I$ in $\Delta$ on which for $i\le n$, $r_i{|_I}=\ep_i$. Fix such an $I$ and consider the subsequence $(x_{k,l})$ that is `supported' on $I$, that is, ${\rm supp}x_{k,l}|_{\Delta}\subseteq I$. A further subsequence { has pairwise disjoint support and a further subsequence of that is weak Cauchy. Thus }the corresponding difference sequence is weakly null. The difference sequence has norm in $[1,2]$ and take values {$-1, 0,1$ on $I$}. 

Now consider that this has been done for all $2^n$ such $I$'s. Label the sequences as $(d^i_j)_{j=1}^{\infty}$ for $i\le 2^n$. By the asymptotic model hypothesis (applied to the weakly null array  $(d^i_j)_{j=1}^{\infty}$, $i\le 2^n$) we can form $r_{n+1}=\sum_{i=1}^{2^n}d^i_{j_i}$ with $1\le \|r_{n+1}\|\le 2K$.  

{ If $(a_n)_{n=1}^N\subseteq \R$ we choose an interval $I\subseteq \Delta$ such that $r_n|_I=sign(a_n)$ for all $1\leq n\leq N$.  Thus, $\|\sum_n a_n r_n\|\ge \|\sum_n a_n r_n|_I\|_{C(I)|}= \sum_n |a_n|$.
Thus $(r_n)$ is a seminormalized sequence which dominates the unit vector basis of $\ell_1$. This contradicts that $\ell_1$ does not embed into $X$ and hence $X^*$ must be separable.} By Zippin's theorem $X$ embeds into a space $Y$ with a shrinking basis $(y_i)$.

ii)  We proceed to show that $X$ is an asymptotic-$c_0$ space with respect to the basis $(y_i)$. We need to prove that there exists a constant $C$ such that for all $n$ every asymptotic space $(e_i)_1^n\in\{X\}_n$ is $C$-equivalent to the unit vector basis of $\ell_{\infty}^n$. If $(e_i)_1^n\in\{X\}_n$ then { also } $(\ep_i e_i)_1^n\in\{X\}_n$ for all sequence of signs $(\ep_i)_1^n$. Therefore, it is sufficient to show that there exists $C$ such that for all { $n\in\N$ and for} every asymptotic space $(e_i)_1^n\in\{X\}_n$ we have $\|\sum_{i=1}^n e_i\|\le C$.

Suppose this is not the case. Then for all $C\ge 1$ there exists $n$ and a normalized asymptotic tree (i.e., countably branching block tree) { $(x_{\alpha})_{\alpha\in T_n}$} in $X$ so that for every branch $\beta=(x_i)_{i=1}^n$ of { $(x_{\alpha})_{\alpha\in T_n}$} there exists $f_{\beta}\in S_{X^*}$ with $f_{\beta}(\sum_{i=1}^n x_i)>C$.

We will construct weakly null seminormalized sequences $(y_i^1)_{i\ge 1}, (y_i^2)_{i\ge 2}, \ldots, (y_i^n)_{i\ge n}$ from the linear combinations of carefully chosen nodes  of { $(x_{\alpha})_{\alpha\in T_n}$} so that, after passing to subsequences in each and relabeling, the array $\{y_i^k: 1\le k\le n, i\ge 1\}$ satisfies $\|y_i^k\|\le K$ for all $1\le k\le n$, $i\ge 1$ and  $\|\sum_{k=1}^n y_{i_k}^k\|>C$ for all $i_1<\cdots <i_n$. This will contradict the assumption that all asymptotic models generated by weakly null arrays are $K$-equivalent to the {unit vector basis of $c_0$.} 

We first describe a general procedure of extracting an array of weakly null sequences from a tree. The actual array will be obtained by { applying this procedure} to a carefully pruned tree (using our assumptions) that we describe { later}.

\noindent{\bf Extracting arrays from trees.}\  The main idea of the construction is that each $y^k_i$ is chosen to be a linear combination of nodes of { $(x_{\alpha})_{\alpha\in T_n}$} from { the } $k$th level so that for every $i_1<\cdots <i_n$ the union of { the} supports (with respect to the tree $T_n$) of $y^1_{i_1}, \ldots, y^n_{i_n}$ contains a (unique) full branch of the tree { $T_n$}.    

 Let { $(x_{\alpha})_{\alpha\in T_n}$ be the tree above. For $(i_1,\ldots,i_k)\in T_n$ we label the node $x^k(i_1, \ldots, i_k):=x_{(i_1,\ldots,i_k)}$. The superscript (which denotes the $k$th level in the tree)} is redundant but we keep it for the sake of clarity.

We will construct the desired $n$-array so that all rows $(y^k_i)_{i\ge k}$ and all diagonal sequences $(y^k_{i_k})_{k=1}^n$, $i_1<\ldots<i_n$ are blocks sequences. We will often prune the tree {$(x_{\alpha})_{\alpha\in T_n}$} by deleting nodes and then relabel the remaining nodes. The pruned tree will always be a full (sub)-tree. Moreover, to ease the notation for later constructions we will relabel the full subtree to match the indices so that the resulting array will have the property that for every diagonal sequence  $(y^k_{i_k})_{k=1}^n$, $i_1<\ldots<i_n$ the corresponding unique full branch is $(x^1(i_1), x^2(i_1, i_2), \ldots, x^n(i_1, i_2, \ldots, i_n))$.

The array is to be labelled as follows and constructed in diagonal order.

\[
\begin{array}{cccccc}
y^1_1 & y^1_2 & y^1_3 & y^1_4 & y^1_5 & \cdots \\
 & y^2_2 & y^2_3 & y^2_4 & y^2_5 & \cdots \\
 & &\ddots &  \ddots &\ddots \\
 &  &  & y^n_n & y^n_{n+1} & \cdots
  \end{array}
\]

Let $y^1_i=x^1(i)$ for all $i\ge 1$. So $(y^1_i)_i$ is the sequence of initial nodes of $(x_{\alpha})_{\alpha\in T_n}$. For the first diagonal sequence $(y^1_1, \ldots, y^n_n)$ take the leftmost branch of $(x_{\alpha})_{\alpha\in T_n}$, that is, 
\begin{equation}\label{first diagonal}
y^1_1=x^1(1), \ y^2_2=x^2(1,2),\  \ldots,\  y^n_n= x^n(1, \ldots, n). 
\end{equation}

{ The node $y^2_3$ will be a sum of two successors to }the nodes $x^1(1)$ and $x^1(2)$ that comprise $y^1_1$ and { $y^1_2$}, respectively. To do this we pick $i_1>2$ and $i_2>2$ large enough so that  $x^2(1, i_1)$ and $x^2(2, i_2)$ are 
supported after $x^1(2)$ (and hence after $x^1(1)$). Delete the nodes $x^2(1, j)$ for $2<j<i_1$ and the nodes $x^2(2, j)$ for $3\le j<i_2$ and relabel the remaining sequences so that the chosen nodes becomes $x^2(1, i_1)=x^2(1, 3)$ and $x^2(2, i_2)=x^2(2, 3)$. Put
\begin{equation}\label{y^2_3}
y^2_3= x^2(1, 3)+ x^2(2, 3).
\end{equation}

We proceed in similar fashion so that each vector {$y^k_j$} of the $k$th row ($j>k>1$) is defined { as a sum of nodes from the $k$th level of the tree $(x_\alpha)_{\alpha\in T_n}$} and which are successors to the nodes that comprise the previously chosen vectors { $y^{k-1}_{k-1},y^{k-1}_{k},...,y^{k-1}_{j-1}$}.  We pick the nodes so that the block conditions are satisfied and relabel the tree after deleting finitely many nodes. Thus $y^3_4$ { is a sum of} nodes successor to the nodes of $y^2_2$ and $y^2_3$ and after relabeling the nodes it becomes 

\begin{equation}\label{y^3_4}
y^3_4= x^3(1, 2, 4)+x^3(1, 3, 4) + x^3(2, 3, 4).
\end{equation}

In general, suppose that  $y^{k-1}_j$ for $k-1\le j<i$ and $k\le n$ are defined. Let { $$y^{k-1}_j=x^{k-1}(\bar t_1)+ x^{k-1}(\bar t_2)+ \ldots=\sum_{m\in A^{k-1}_j}x^{k-1}(\bar t_m) \quad\textrm{ for some }A^{k-1}_j\subset \N$$ be the enumeration of the (finitely many) nodes comprising $y^{k-1}_j$'s in the order they appear and where each $\bar t_s$ is  a $k-1$-tuple with maximal entry $j$.

We denote concatenation by $(a_1,\ldots, a_n)\con a_{n+1}=(a_1,\ldots, a_n,a_{n+1})$.
By passing to subsequences  and relabeling  the sequences of successor nodes $$(x^k(\bar t_1\con l))_{l\ge j}, (x^k(\bar t_2\con l))_{l>\ge j}, (x^k(\bar t_3\con l))_{l\ge j}, \ldots$$we may assume each of these vectors are supported after the previously chosen ones.  We define $y^k_i$ as a sum of successors to the nodes comprising $y^{k-1}_{k-1}, \ldots, y^{k-1}_{i-1}$. That is, we put 
\begin{equation}\label{y^k_i}
y^k_i=\sum_{j=k-1}^{i-1}\sum_{m\in A^{k-1}_j} x^k(\bar t_m\con i).
\end{equation}
Note that $j$ is the maximal entry of $\bar t_m\in A^{k-1}_j$ and hence $x^k(\bar t_m\con i)$ is a successor of $x^k(\bar t_m)$ as $j<i$.}

This completes the construction of the array. It { follows } that the support of any diagonal sequence  $(y^k_{i_k})_{k=1}^n$, $i_1<\ldots<i_n$ contains the unique full branch $(x^1(i_1), x^2(i_1, i_2), \ldots, x^n(i_1, i_2, \ldots, i_n))$ as desired.

{\bf Pruning the tree.}
For notational convenience we will denote branches $$\beta=(x^1(i_1), x^2(i_1, i_2), \ldots, x^n(i_1, i_2, \ldots, i_n))$$ of the tree by $\beta=(i_1, i_2, \ldots, i_n)$. From the construction the support of (the sum of) each sequence $y^1_{i_1}, \ldots, y^n_{i_n}$ consists of the unique full branch $\beta=(i_1, i_2, \ldots, i_n)$ and other off-branch nodes whose numbers add up quickly as $i_n$ gets large. 
By our assumption there is a branch functional $f_{\beta}$ so that
\begin{equation}\label{assumption}
f_{\beta}\Big(\sum_{k=1}^n x^k(i_1, \ldots, i_k)\Big)>C.
\end{equation} Our goal here is to show that for all $\ep>0$ we can prune the tree so that the array (with respect to the pruned tree) satisfies
\begin{equation}\label{bounded}
\|y^k_i\|\le K{ + \ep}, \ \text{for all}\ 1\le k\le n,\ i\ge k,
\end{equation}
and

\begin{equation}\label{small off-diagonal}
f_{\beta}\Big(\sum_{k=1}^n y^k_{i_k}\Big)\ge C-\ep.
\end{equation}

Let $\ep>0$. Fix $(\ep_k)_{k=1}^n$ so that $\sum_{k=1}^n \ep_k<\ep$. Let {$(x_{\alpha})_{\alpha\in T_n}$} be a full subtree satisfying block conditions described in { the} above construction. That is, every sequence of successor nodes of {$(x_{\alpha})_{\alpha\in T_n}$} is a block basis and whenever $y^k_i$ is defined as in (\ref{y^k_i}) the sequences $(y^1_{i_1}, \ldots, y^n_{i_n})$ are blocks as well. 

As before we will proceed in diagonal order (of the array). Let $y^1_i=x^1(i)$ for all $i\ge 1$. For the first diagonal sequence $(y^1_1, \ldots, y^n_n)$ again we take the leftmost branch of {$(x_{\alpha})_{\alpha\in T_n}$}, that is, 
\begin{equation}
y^1_1=x^1(1), \ y^2_2=x^2(1,2),\  \ldots,\  y^n_n= x^n(1, \ldots, n). 
\end{equation}
The condition (\ref{bounded}) is clearly satisfied since the tree is normalized and the condition (\ref{small off-diagonal}) follows from the assumption (\ref{assumption}).

We wish to define $y^2_3$ as in (\ref{y^2_3}).  This will require two steps. First consider the sequences of level 2 successor nodes $$(x^2(1, l))_{l\ge 3}, (x^2(2, l))_{l\ge 3}.$$
By our main assumption, the array formed by these sequences can be refined to generate an asymptotic model $K$-equivalent to { the unit vector basis of $\ell_\infty^2$.} Thus by passing to subsequences, relabeling and ignoring tiny perturbations we can assume that for all $3\le l_1<l_2$,
\begin{equation}\label{as-model-second-level}
\|x^2(1, l_1)+x^2(2, l_2)\|\le K.
\end{equation}
This will ensure that whenever $y^2_3$ is defined as in (\ref{y^2_3}) the condition (\ref{bounded}) is satisfied. 
The second refinement towards ensuring (\ref{small off-diagonal}) is somewhat more complicated. 

Consider again the sequences of successor nodes $(x^2(1, l))_{l\ge 3}$ and $(x^2(2, l))_{l\ge 3}$. By the main assumption each of these sequences {generate spreading models which are $K$-equivalent to the unit vector basis of $c_0$}.  Fix $N\ge 1+K^2/\ep^2_1+2K/\ep_1.$ By passing to subsequences and relabeling we can assume that both $(x^2(1, l))^{N+3}_{l=3}$ and $(x^2(2, l))^{N+3}_{l=3}$ are $K$-equivalent to the unit vector basis of $\ell_{\infty}^N$.  For every branch $\beta=(i_1, \ldots, i_n)$ of $T_n$ { we let} $f_{(i_1, i_2, \ldots, i_n)}$ denote the corresponding branch functional satisfying (\ref{assumption}).
For each $3\le l\le N+3${,  $f_{(1,l)\Con \bar j}$ and $f_{(2, l)\Con \bar j}$ are the branch functionals for branches extending $(1, l)$ and $(2,l)$ respectively, where $\bar j$ is an $(n-2)$-tuple.} We stabilize the values of these functionals on the chosen nodes. That is, by passing to subsequences and { ignoring} tiny perturbations we can assume that for all $\bar j, \bar j'$ we have{
$$f_{(1, l)\Con \bar j}(x^2(2, t))=f_{(1, l)\Con\bar j'}(x^2(2, t))\ \text{and} \ f_{(2, l)\Con \bar j}(x^2(1, t))=f_{(2, l)\Con \bar j'}(x^2(1, t)),$$ }for all $3\le l, t\le N+3.$

{\bf Claim.}\ There exist $3\le l_1, l_2\le N+3$ so that for all $\bar j$
\begin{equation}\label{gap-base} |f_{(1, l_1)\Con \bar j}(x^2(2, l_2))|<\ep_1\ \text{and}\ |f_{(2, l_2)\Con \bar j}(x^2(1, l_1))|<\ep_1.\end{equation}

For any functional $f$ of norm at most 1 and sequence $(x_t)_{t=1}^n$ which is $K$-equivalent to {the} unit vector basis of $\ell_{\infty}^n$ there is a sequence of signs $\delta_t=\pm 1$ so that

\begin{equation}\label{K/ep}
\sum_{t=1}^{n}|f(x_t)|=\left|f\Big(\sum_{t=1}^{n}\delta_t x_t\Big)\right|\le K.
\end{equation}

It follows that the cardinality {$|\{t: |f(x_t)|\ge \ep_1\}|\le K/\ep_1$. Thus
for each $l$ and $\bar j$,}  
$$|A_l|:=\Big|\{t: |f_{(1, l)\Con \bar j}(x^2(2, t))|< \ep_1\}\Big|\ge N-K/\ep_1.$$
Then for any $B\subset \{3, \ldots, N+3\}$ with $K/\ep_1+1\le |B|<K/\ep_1 +2$ we have 
$$\Big|\bigcap_{l\in B} A_l\Big|\ge 1.$$
Indeed, $N-|B|K/\ep_1\ge N-K^2/\ep^2_1-2K/\ep_1\ge 1$. Fix such a subset $B$ and let $l_2\in \bigcap_{l\in B} A_l$.  Then $|f_{(1, l)\Con \bar j}(x^2(2, l_2))|<\ep_1$ for all $l\in B$. Now consider the functionals $f_{(2, l_2)\Con \bar j}$. Since $(x^2(1, l))_{l\in B}$ is $K$-equivalent to {the unit vector basis of} $\ell_{\infty}^{|B|}$ and $|B|\ge K/\ep_1+1$, by a similar argument as above, there is $l_1\in B$ such that
$|f_{(2, l_2)\Con \bar j}(x^2(1, l_1))|<\ep_1$, proving the claim.

Now we relabel the nodes as $x^2(1,l_1)=x^2(1, 3)$ and $x^2(2, l_2)=x^2(2, 3)$ (by deleting finitely many nodes) and put
\begin{equation}
y^2_3=x^2(1, 3)+x^2(2,3).
\end{equation}

{At this stage the pruned tree has the following {\em gap property} of the branch functionals $f_{(1,3)\con\bar j}$ and $f_{(2,3)\con\bar j}$. 
\begin{eqnarray*}
f_{(1,3)\Con\bar j}\Big((y^1_1 + y^2_3)-(x^1(1) + x^2(1, 3))\Big)=f_{(1,3)\Con\bar j}\Big(x^2(2,3)\Big)<\ep_1,\\
f_{(2,3)\Con\bar j}\Big((y^1_2 + y^2_3)-(x^1(2) + x^2(2, 3))\Big)=f_{(2,3)\Con\bar j}\Big(x^2(1,3)\Big)<\ep_1
\end{eqnarray*}}
{ We have that $x^1(1)$ and $x^2(1,3)$ are the nodes on the branch of $(1, 3)\con \bar j$. Thus the first above inequality states that the branch functional $f_{(1,3)\Con \bar j}$ is small on the off branch part of $y^1_1+y^2_3$, and the second above inequality states that the branch functional $f_{(2,3)\Con \bar j}$ is small on the off branch part of $y^1_2+y^2_3$.  This will be important for us as the branch functionals $f_\beta$ are defined to be large on their branch.  We will eventually be able to obtain \eqref{small off-diagonal} by showing that $f_\beta$ is greater than $C$ on the  branch part of $\sum_{k=1}^n y^k_{i_k}$ and $f_\beta$ is smaller than $\ep$ on the off branch part of $\sum_{k=1}^n y^k_{i_k}$ where $\beta=(i_1,\ldots,i_n)$.}

For the sake of clarity we also show how to define $y^3_4$ as in (\ref{y^3_4}) before proceeding with the inductive step. The array formed by the sequences of level 3 successor nodes $$(x^3(1, 2, l))_{l\ge 4}, (x^3(1, 3, l))_{l\ge 4}, (x^3(2, 3, l))_{l\ge 4}$$
can be refined to generate an asymptotic model $K$-equivalent to {the unit vector basis of $\ell_\infty^3$}. Thus by passing to subsequences, relabeling and ignoring tiny perturbations we get that for all $4\le l_1<l_2<l_3$,
\begin{equation}\label{as-model-third-level}
\|x^3(1, 2, l_1)+ x^3(1, 3, l_2) + x^3(2, 3, l_3)\|\le K.
\end{equation}
This will ensure the condition (\ref{bounded}).

The second refinement is done { as before}. Fix a large $N=N(K,\ep_2/2)$ and using the $c_0$ spreading models assumption pick sequences  $(x^3(1, 2, l))_{l=4}^{N+4}$, $(x^3(1, 3, l))_{l=4}^{N+4}$, and $(x^3(2, 3, l))_{l=4}^{N+4}$ that are $K$-equivalent to the unit vector basis of $\ell_{\infty}^N$. Refine the tree by passing to subsequences of the successors of these so that the branch functionals $f_{(1, 2, l)\Con \bar j}$, $f_{(1, 3, l)\Con \bar j}$, and $f_{(2, 3, l)\Con \bar j}$ are stabilized. That is, their values on the chosen nodes are independent of $\bar j$. Then a similar combinatorial argument as before yields (see the Gap lemma below) a node from each sequence which we relabel as $x^3(1, 2, 4)$, $x^3(1, 3, 4)$, and $x^3(2, 3, 4)$ so that 
$$|f_{(1, 2, 4)\Con \bar j}(x^3(1, 3, 4))|<\ep_2/2,\ |f_{(1, 2, 4)\Con \bar j}(x^3(2, 3, 4))|<\ep_2/2,$$

$$|f_{(1, 3, 4)\Con \bar j}(x^3(1, 2, 4))|<\ep_2/2,\ |f_{(1, 3, 4)\Con \bar j}(x^3(2, 3, 4))|<\ep_2/2, \ \text{and}$$

$$|f_{(2, 3, 4)\Con \bar j}(x^3(1, 2, 4))|<\ep_2/2,\ |f_{(2, 3, 4)\Con \bar j}(x^3(1, 3, 4))|<\ep_2/2.$$

Let 
$$y^3_4=x^3(1, 2, 4)+x^3(1, 3, 4)+x^3(2, 3, 4).$$ Then 
the branch functionals through these nodes satisfy the desired gap properties: { For $1\leq t_1<t_2<4$, denoting} ${\mathbf x}_{(t_1, t_2, {4})}=x^1(t_1)+x^2(t_1, t_2)+x^3(t_1, t_2, {4})$ and ${\mathbf y}_{(t_1, t_2, {4})}=y^1_{t_1}+y^2_{t_2}+y^3_{{4}}$ we have

\begin{eqnarray*}
\left|f_{(1, 2, 4)\Con \bar j}\Big({\mathbf y}_{(1, 2, 4)}-{\mathbf x}_{(1, 2, 4)} \Big)\right|
&\le& \Big|f_{(1, 2, 4)\Con \bar j}\big(x^3(1, 3, 4) \big)\Big|+\Big|f_{(1, 2, 4)\Con \bar j}\big(x^3(2, 3, 4)\big)\Big|\\
&<&\ep_2/2 +\ep_2/2,
\end{eqnarray*}
\begin{eqnarray*}
&&\left|f_{(1, 3, 4)\Con \bar j}\Big({\mathbf y}_{(1, 3, 4)}-{\mathbf x}_{(1, 3, 4)} \Big)\right|\\
&\le&\Big|f_{(1, 3, 4)\Con \bar j}\big(x^2(2, 3) \big)\Big|
+\Big|f_{(1, 3, 4)\Con \bar j}\big(x^3(1, 2, 4) \big)\Big|+\Big|f_{(1, 3, 4)\Con \bar j}\big(x^3(2, 3, 4)\big)\Big|\\
&<&\ep_1 +\ep_2/2 +\ep_2/2,
\end{eqnarray*}
and
\begin{eqnarray*}
&&\left|f_{(2, 3, 4)\Con \bar j}\Big({\mathbf y}_{(2, 3, 4)}-{\mathbf x}_{(2, 3, 4)} \Big)\right|\\
&\le&\Big|f_{(2, 3, 4)\Con \bar j}\big(x^2(1, 3) \big)\Big|
+\Big|f_{(2, 3, 4)\Con \bar j}\big(x^3(1, 2, 4) \big)\Big|+\Big|f_{(2, 3, 4)\Con \bar j}\big(x^3(1, 3, 4)\big)\Big|\\
&<&\ep_1 +\ep_2/2 +\ep_2/2,
\end{eqnarray*}

{  As before, the idea is that $f_{\beta}$ is large on the branch part of $y^1_{t_1} + y^2_{t_2}+y^3_{4}$ and is small on the off branch part where $\beta=(t_1,t_2,4)$. }

{We now} proceed inductively. Suppose that for $k-1\le j<i$ and $k\le n$,
$$y_j^{k-1}=\sum_{m\in A^{k-1}_j}x^{k-1}(\bar t_m)$$ are defined where $x^{k-1}(\bar t_m)$ are $(k-1)$-level nodes and { $A^{k-1}_j\subset \N$ is finite}.  For each $\bar t_m=(t_1, \ldots, t_{k-1})$ denote the sum of the  initial segment of a diagonal sequence of the array constructed thus far by
$${\mathbf y}_{\bar t_m}=\sum_{i=1}^{k-1}y^i_{t_i},$$ and the sum of the initial segment of the tree by
$${\mathbf x}_{\bar t_m}=\sum_{i=1}^{k-1}x^i(t_1, \ldots, t_{k-1}).$$

For the induction hypothesis we also assume that the branch functionals $f_{\bar t_m\Con \bar j}$ for the branches whose initial segments are $\bar t_m$ satisfy the gap {property}: 

\begin{equation}\label{gap induction}
\Big|f_{\bar t_m\Con \bar j}({\mathbf y}_{\bar t_m}-{\mathbf x}_{\bar t_m})\Big|<\sum_{i=1}^{k-1}\ep_i.
\end{equation}

Consider the array formed by the sequences of successor nodes $$(x^k({\bar t_1\con l}))_{l>\max \bar t_1}, (x^k({\bar t_2\con l}))_{l>\max \bar t_2}, \ldots,  (x^k({\bar t_M\con l}))_{l>\max \bar t_M}$$ for $m\in \bigcup_{j=k-1}^{i-1}A^{k-1}_j$ and where $M=|\bigcup_{j=k-1}^{i-1}A^{k-1}_j|$. The array is formed in the order the nodes appear in the support of $y^{k-1}_{k-1}, \ldots, y^{k-1}_{i-1}$. By the main assumption the array generates an asymptotic model $K$-equivalent to the unit vector basis {of $\ell_\infty^M$}.  Thus by passing to subsequences and relabeling we can assume that for all $\max_{1\le m\le M} \max \bar t_m<l_1<l_2<\ldots<l_M$, 

\begin{equation}\label{induction bounded}
\left\|\sum_{m=1}^M x^k({\bar t_m\con l_m})\right\|\le K.
\end{equation}

Fix a large $N=N(K, \ep_k/M)$ (determined by the lemma below). For each $1\le m\le M$, using the fact that every sequence of successor nodes generates a $c_0$ spreading model, pick $\big(x^k{(\bar t_m\con l)}\big)_{l\in B_m}$, $|B_m|=N$, which is $K$-equivalent to the unit vector basis of $\ell^{N}_{\infty}$. For all $m$ and $l\in B_m$, by passing to a subsequence of the successors $(x^{k+1}{(\bar t_m\con l\con j)})_j$ of $x^k({\bar t_m\con l})$ we can assume that all the branch functionals $f_{\bar t_m\Con l\Con \bar j}$ are stabilized on the chosen nodes. That is, for all $\bar j$ and $\bar j'$, ignoring tiny perturbations, we have
$$f_{\bar t_m\Con l\Con \bar j}(x^k{(\bar t_{m'}\con l')})=f_{\bar t_m\Con l\Con \bar j'}(x^k{(\bar t_{m'}\con l')})$$ for all $m\neq m'$, and $l\in B_m, l'\in B_{m'}$. (Note: If $k=n$, the last level of the tree, then all the branch functionals are already determined.)

{\bf Claim.}\ For all $1\le m\le M$ there exist $l_m\in B_m$ such that for all $m\neq m'$
\begin{equation}\label{gap-induction}\Big|f_{\bar t_m\Con l_m\Con \bar j}(x^k{(\bar t_{m'}\con l_{m'})})\Big|<\ep_{k}/M.
\end{equation}

This is consequence of the following combinatorial lemma (for $\ep=\ep_k/M$.) 

\begin{Gap} Let $\ep>0$, $M\in \N$. Then there exists $N=N(\ep, M, K)$ such that given sequences $(x^1_l)_{l=1}^N, \ldots, (x^M_l)_{l=1}^N$  each $K$-equivalent to the unit vector basis of $\ell_{\infty}^N$ and functionals  $(f^1_l)_{l=1}^N, \ldots, (f^M_l)_{l=1}^N$ of norm at most 1 there exists $l_1, \ldots, l_M$ such that $$\Big|f^j_{l_j}(x^i_{l_i})\Big|<\ep,\ \text{for all}\ i\neq j.$$
\end{Gap}
\begin{proof}
The proof is by induction on $M$. For the base case $M=2$ we prove the following which is a slight generalization of (\ref{gap-base}): For all $N_0\in\N$ there exists $N=N(N_0, \ep, K)$ so that whenever $(x^1_l)_{l=1}^N, (x^2_l)_{l=1}^N$ and $(f^1_l)_{l=1}^N, (f^2_l)_{l=1}^N$ are as in the statement there exist $A_1, A_2\subset \{1, \ldots, N\}$ with $|A_1|, |A_2|\ge N_0$ such that for all $j\in A_1$ and $i\in A_2$ we have $|f^1_j(x^2_i)|<\ep$ and $|f^2_i(x^1_j)|<\ep$.

Fix $N\ge N_0(1+K/\ep +K^2/\ep^2)$. For any functional $f$ of norm at most 1 and sequence $(x_i)_1^n$ $K$-equivalent to the unit vector basis $\ell_{\infty}^n$ we have, by (\ref{K/ep}), $|\{i:|f(x_i)|\ge \ep\}|\le K/\ep$. 
Thus for $N_0(1+K/\ep)\le N_1\le N_0(1+K/\ep)+1$,
$$\Big|\bigcap_{l=1}^{N_1}\Big\{1\le i\le N:|f^1_l(x^2_i)|< \ep\Big\} \Big|\ge N-N_1 K/\ep\ge N_0.$$
Let $A_2$ be a subset of $\bigcap_{l=1}^{N_1}\big\{1\le i\le N:|f^1_l(x^2_i)|< \ep\big\}$ with cardinality $N_0$ and we have
$$|A_1|=\Big|\bigcap_{l\in A_2} \Big\{1\le i\le N_1: |f^2_l(x^1_i)|<\ep\Big\}\Big|\ge N_1-N_0 K/\ep\ge N_0,$$ as desired. 

For the induction suppose that for all $N_0\in\N$ there exists $N$ and $A_1, \ldots, A_m$ with $|A_i|\ge N_0$ so that for all $l_i\in A_i$
$$\Big|f^j_{l_j}(x^i_{l_i})\Big|<\ep,\ \text{for all}\ 1\le i\neq j\le m.$$
Fix $N_0\ge 1+K/\ep +K^2/\ep^2$ and apply the argument in the base case for the pairs $(x^i_l)_{l\in A_i}, (x^{m+1}_l)_{l=1}^{N_0}$ and $(f^i_l)_{l\in A_i}, (f^{m+1}_l)_{l=1}^{N_0}$ for $1\le i\le m$ to get the desired $(m+1)-$tuple $l_1, \ldots, l_{m+1}$ so that
$$\Big|f^j_{l_j}(x^i_{l_i})\Big|<\ep,\ \text{for all}\ 1\le i\neq j\le m+1.$$
\end{proof}

Consider $l_1, \ldots, l_M$ from the Claim. We discard the nodes $x^l{(\bar t_m\con l)}$, $l\in B_m$ and $l\neq l_m$ and relabel the rest so that for all $1\le m\le M$, $x^k{(\bar t_m\con l_m)}=x^k{(\bar t_m\con i)}$ where $i=\max_m \max \bar t_m+1$ and put

\begin{equation}\label{y^k_i}
y^k_i=\sum_{j=k-1}^{i-1}\sum_{m\in A^{k-1}_j} x^k({\bar t_m\con i)}.
\end{equation}

By (\ref{induction bounded}) $\|y^k_i\|\le K$. By the induction hypothesis (\ref{gap induction}) and the Claim (\ref{gap-induction}) we have

\begin{equation}
\Big|f_{(\bar t_m, i, \bar j)}\Big(({\mathbf y}_{\bar t_m}+y^k_i)-({\mathbf x}_{\bar t_m}+x^k{(\bar t_m\con i)})\Big)\Big|<\sum_{i=1}^{k-1}\ep_i +\sum_{m=1}^M \ep_k/M= \sum_{i=1}^k \ep_i
\end{equation}
for all $1\le m\le M$ and $\bar j$, as desired. This concludes the construction of the array. 

Now let $\beta=(i_1, \ldots, i_n)$ be arbitrary. Then by the construction and our main assumption we have 
\begin{equation*}
f_{\beta}\Big(\sum_{k=1}^n y^k_{i_k}\Big)\ge f_{\beta}\Big(\sum_{k=1}^n x^k(i_1, \ldots, i_k)\Big)- \Big|f_{\beta}\Big(\sum_{k=1}^n x^k(i_1, \ldots, i_k)-\sum_{k=1}^n y^k_{i_k}\Big)\Big|  \ge C-\sum_{k=1}^n \ep_k>C-\ep.
\end{equation*}
The proof is completed.

\end{proof}

\end{document}